\documentclass[11pt]{article}

\usepackage[top=50mm, bottom=50mm, left=50mm, right=50mm]{geometry}
\usepackage{lineno}
\usepackage{indentfirst}
\usepackage{bm}
\usepackage{amssymb}
\usepackage{amsmath}
\usepackage{amsthm}
\usepackage{algorithm}
\usepackage{algorithmicx}
\usepackage{algpseudocode}
\usepackage{appendix}
\usepackage{cases}
\usepackage{epsfig}
\usepackage{graphicx}
\usepackage{graphics}
\usepackage{float}
\usepackage{subfigure}
\usepackage{multirow}
\usepackage{mathrsfs}
\usepackage{color}
\usepackage{fullpage}
\usepackage[normalem]{ulem}
\usepackage{makeidx}
\usepackage{xspace}
\usepackage{wrapfig}
\usepackage{cite}
\usepackage{tabularx}
\usepackage{latexsym}
\usepackage{amsfonts}
\usepackage{epstopdf}
\makeindex

\newtheorem{claim}{Claim}

\newtheorem{thm}{Theorem}[section]

\newtheorem{lem}[thm]{Lemma}

\newtheorem{obs}[thm]{Observation}

\definecolor{darkred}{rgb}{1, 0.1, 0.3}
\definecolor{darkblue}{rgb}{0.1, 0.1, 1}
\definecolor{darkgreen}{rgb}{0,0.6,0.5}

\newcommand {\mm}[1] {\ifmmode{#1}\else{\mbox{\(#1\)}}\fi}

\floatname{algorithm}{\textbf{Algorithm:}}

\makeatletter
\def\thanks#1{\protected@xdef\@thanks{\@thanks
        \protect\footnotetext{#1}}}
\makeatother

\begin{document}

\title{Saturation Numbers for Linear Forests $P_7+tP_2$
\thanks{$^{*}$ Corresponding authors.
\newline\indent\indent\indent\small \emph{Email address:} 23121715@bjtu.edu.cn (Y. Zhang), rxhao@bjtu.edu.cn (R.-X. Hao), zhenhe@bjtu.edu.cn (Z. He), zhuwenhanzwh@163.com (W.-H. Zhu)}
}
\author{Yu Zhang, Rong-Xia Hao, Zhen He$^{*}$,
Wen-Han Zhu
\\
{\it\small  School of Mathematics and Statistics, Beijing Jiaotong University, Beijing 100044, P.R. China.}
}

\date{}

\maketitle

\begin{abstract}

Let $H$ be a fixed graph, a graph G is $H$-saturated if it has no copy of $H$ in $G$, but the addition of any edge in $E(\overline G)$  to $G$ results in an $H$-subgraph. The saturation number sat$(n,H)$ is the minimum number of edges in an $H$-saturated graph on $n$ vertices. In this paper, we determine the saturation number sat$(n,P_7+tP_2)$ for $n\geq \frac {14}{5}t+27$ and characterize the extremal graphs for $n\geq \frac{14}{13}(3t+25)$.

\end{abstract}

\section{Introduction}

Let $G$ be a graph with no loop and multiple edges, $V(G)$ and $E(G)$ be the vertex set and edge set of $G$, respectively. Let $|V(G)|$ (or $|G|$) be the order of $G$.
For any $x\in V(G)$, the set $\{y\in V(G)\setminus \{x\}:xy\in E(G)\}$ is the neighborhood of $x$ in $G$, denoted by $N_G(x)$. The minimum degree of $G$ is the parameter $\delta (G)=$ min $\{ d_G(x): x\in V(G)\}$, where $d_G(x)=|N_G(x)|$ is the degree of $x$ in $G$.
Let $N_G[x]=N(x)\cup \{v\}$, $\overline G$ be the complement of $G$, $C_n$ be a cycle of order $n$, $K_n$ be a complete graph of order $n$ and $P_k=x_1x_2\dots x_k$ be a path of $k$ vertices. The graph $G+H$ means the $disjoint\ union$ of $G$ and $H$. Specially, $tH$ denotes the $disjoint\ union$ of $t$ copies of $H$. Denoted by $V_i(G)$ the set of vertices with degree $i$ in $G$. The subscript $G$ in $N_G(x)$, $d_G(x)$ and $N_G[x]$ will be omitted if $G$ is clear from the context. Define $\alpha_k(G)=$ max $\{m:P_k+mP_2\subset G\}$.

A $fan$ $F_i$ consists of $i$ triangles sharing one vertex, say $x$. A $ffan$ $F^2_{i,j}$ is the graph $F_i+F_j+xy$, where $F_i$ and $F_j$ are two disjoint fans that all triangles of $F_i$ share one vertex $x$ and all triangles of $F_j$ share one vertex $y$, see a representation in Figure \ref{Fig.1}. A $\Delta_+ fan$ $\Delta_+ F_i$ consists of $i-1$ triangles and a $K_4$ sharing one vertex $x$, which is shown in Figure \ref{Fig.2}. A $\Delta fan$ $\Delta F_i$ obtained from $\Delta_+ F_i$ by deleting an edge which is adjacent to $x$ in $K_4$, see Figure \ref{Fig.3}. Moreover, we denote $x$ as a center vertex in each of $F_i$, $F^2_{i,j}$, $\Delta_+ fan$ and $\Delta fan$. A graph $G$ is $book$-$structural$ if there exist $u$, $v\in V(G)$ such that $d(u)=d(v)=2$ and $N(u)=N(v)$. An example is shown in Figure \ref{Fig.4}.
\begin{figure}[htbp]
  \centering
 \begin{minipage}[t]{0.36\textwidth}
    \centering
    \includegraphics[width=\textwidth]{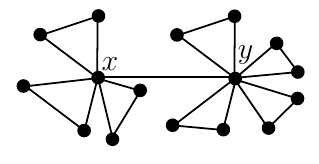} 
    \caption{$F^2_{3,4}$}\label{Fig.1}
  \end{minipage}
\hfill
\centering
 \begin{minipage}[t]{0.22\textwidth}
    \centering
    \includegraphics[width=\textwidth]{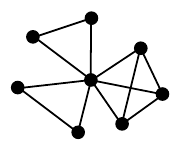} 
     \caption{$\Delta_+ F_3$}\label{Fig.2}
  \end{minipage}
  \hfill
\centering
\begin{minipage}[t]{0.22\textwidth}
    \centering
    \includegraphics[width=\textwidth]{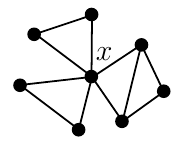} 
     \caption{$\Delta F_3$}\label{Fig.3}
  \end{minipage}
  \hfill
\end{figure}
\begin{figure}[htbp]
\begin{center}
\scalebox{0.65}[0.65]{\includegraphics{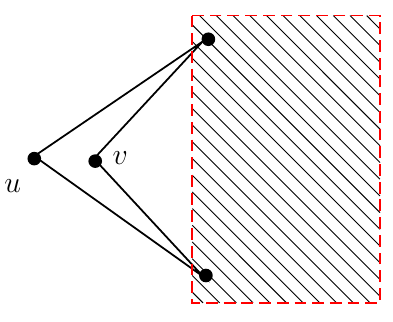}}
\caption{An example of a graph with book-structural}\label{Fig.4}
\end{center}
\end{figure}

Let $H$ be a graph. For any graph $G$, we say that $G$ is $H$-$saturated$ if there exists no copy of $H$ in $G$ but for any $e\in E(\overline G)$, $G+e$ contains a copy of $H$. For $n\geq |V(H)|$, let $sat(n,H)=\min\{|E(G)|:G\ is\ H$-$saturated\ and\ V(G)=n\}$ denote the $saturation\ number$ of $H$, and $Sat(n,H)$ denote the set of $H$-$saturated$ graphs with size $sat(n,H)$.

The result about the saturation number of a graph was introduced by Erd\H{o}s, Hajnal, Moon in \cite{6} in which the authors proved $sat(n,K_t)=\binom {t-2}{2}+(n-t+2)(t-2)$ and $K_{t-2}\vee \overline K_{n-t+2}$ is the unique minimum $K_t$-saturated graph. Most of the results on the saturation numbers of forests concern $linear$ forests. The result of this type concerned a forest consisting of $t$ independent edges that L. K\'{a}szonyi et al. in \cite{1} proved that for $n \geq 3t-3$, sat$(n,tP_2)=3t-3$ and  $Sat(4,2P_2)=\{K_3+K_1,S_4\}$, otherwise, $Sat(n,tP_2)=\{(t-1)K_3+\overline{K}_{n-3t+3}\}$.
There exist a number of recent results on the saturation numbers of specific linear forests in \cite{7,8,9}. Chen et al. in \cite{2} proved that for $n$ sufficiently large, $sat(n,P_3+tP_2)=3t$ and $tK_3+\overline{K}_{n-3t}\in Sat(n,P_3+tP_2)$. Moreover, $sat(n,P_4+tP_2)=3t+7$ and $K_5+(t-1)K_3+\overline{K}_{n-3t-2}\in Sat(n,P_4+tP_2)$. And they also obtained that $\lfloor \frac{n}{2} \rfloor \leq sat(n,tP_3)\leq \lfloor \frac{n}{2} \rfloor+3(t-1)$. He et al. in \cite{10} improved the lower bound on $sat(n,tP_3)$ that $sat(n,tP_3)\geq \lfloor \frac{n}{2} \rfloor+\frac{5}{2}(t-1)$.
 Fan et al. in \cite{3} proved that $sat(n,P_5+tP_2)=$ min$ \{\lceil \frac{5n-4}{6} \rceil,3t+12\}$ and $Sat(n,P_5+tP_2)=\{K_6+(t-1)K_3+\overline{K}_{(n-3t-3)}\}$ for $n\geq \frac{1}{5}(18t+76)$. Yan et al. in \cite{4} proved that $sat(n,P_6+tP_2)=$ min$ \{n-\lfloor \frac{n}{10} \rfloor,3t+18\}$ and $Sat(n,P_6+tP_2)=\{K_7+(t-1)K_3+\overline{K}_{(n-3t-4)}\}$ for $n\geq \frac{10}{3}t+20$.

Our main result proves that the saturation number of $P_7+tP_2$ and characterizes the minimum $(P_7+tP_2)$-saturated graphs, shown in Theorem \ref{thm9}.

\begin{thm}\label{thm9}
Let $n$ and $t$ be two positive integers with $n\geq \frac {14}{5}t+27$. Then

  \item (1) $sat(n,P_7+tP_2)=min\{3t+25,n-\lfloor \frac{n}{14} \rfloor\}$,
  \item (2) $Sat(n,P_7+tP_2)=K_8+cK_3+F+\overline{K}_{n-3t_3-|F|-8}$ for $n\geq \frac{14}{13}(3t+25)$, where $c$ is a non-negative integer, $F$ is the disjoint union of fans with order at least $7$ and $\alpha^{\prime}(cK_3+F)=t-1$.

\end{thm}

\section{Preliminaries and observation}

\begin{lem}\label{lem1}\cite{5} For a graph $G$,

\begin{center}
$\alpha^{\prime}(G)=\frac{1}{2}min\{|G|+|S|-o(G-S):S\subset V(G)\}$,
\end{center}
where $\alpha^{\prime}(G)$ is the matching number of $G$ and $o(G-S)$ is the number of odd components of $G-S$.

\end{lem}

\begin{lem}\label{lem2}\cite{2} Let $k_1,\dots ,k_m\geq 2$ be $m$ integers and $G$ be a $(P_{k_1}+P_{k_2}+\dots +P_{k_m})$-saturated graph. If $d(x)=2$ and $N(x)=\{u,v\}$, then $uv\in E(G)$.

\end{lem}

\begin{lem}\label{lem4}\cite{4} Let $G$ be a $(P_k+tP_2)$-saturated graph with $k\geq 2, t\geq 1$. If $V_0(G)\neq \emptyset$, then $V_1(G)= \emptyset$. Moreover, for any $x\in V(G)\backslash V_0(G)$, we have

\begin{center}
$N_G[x]\cup \{w\}\subset V(H)$,
\end{center}
where $H$ is any copy of $(P_k+tP_2)$ in $G+xw$ and $w$ is a vertex in $V_0(G)$.

\end{lem}

\begin{obs}\label{lem2.4} Let $G$ be a $(P_k+tP_2)$-saturated graph with $k\geq 2, t\geq 1$. Let $Q$ be a component of $G$. Then

  \item (1) $\alpha^{\prime}(Q)< \alpha^{\prime}(Q+uv)$, or $P_k\subset Q$ and $\alpha_k(Q)<\alpha_k(Q+uv)$, or $P_k\not\subset Q$ and $P_k\subset Q+uv$ for any $u,v\in Q$ and $uv\notin E(G)$;
  \item (2) $\alpha^{\prime}(Q+Q_1)< \alpha^{\prime}(Q+Q_1+uv)$, or $P_k\subset Q+Q_1$ and $\alpha_k(Q+Q_1)<\alpha_k(Q+Q_1+uv)$, or $P_k\not\subset Q+Q_1$ and $P_k\subset Q+Q_1+uv$ for any $u\in Q$ and $v\in Q_1$, where $Q_1$ is a component of $G$ with $Q\neq Q_1$.

\end{obs}

\begin{lem}\label{lem15}
Let $G$ be a $(P_k+tP_2)$-saturated graph with $k\geq 2, t\geq 1$ and $|V_0(G)|\geq 2$. For any component $Q$ of $G$ with $|Q|\leq k-1$, then
\item (1) $Q$ is a clique;
  \item (2) $|Q|$ is odd or $|Q|=k-1$ with odd $k$.

\end{lem}

\begin{proof}

(1) By Lemma \ref{lem1}, there exists $W\subset V(Q)$ such that

\begin{center}
$\alpha^{\prime}(Q)=\frac{1}{2}(|Q|+|W|-o(Q-W))$.
\end{center}
Let $Q^{\prime}_1, Q^{\prime}_2, \dots , Q^{\prime}_p$ be the components of $Q-W$.

If there exist two vertices $u,v\in W\cup V(Q^{\prime}_i)$ such that $uv\notin E(Q)$, let $Q^{\prime}=Q+uv$.
Then
\begin{center}
$\alpha^{\prime}(Q^{\prime})\geq \alpha^{\prime}(Q)=\frac{1}{2}(|Q^{\prime}|+|W|-o(Q^{\prime}-W))\geq \alpha^{\prime}(Q^{\prime})$.
\end{center}

Therefore, $\alpha^{\prime}(Q^{\prime})=\alpha^{\prime}(Q)$ and $P_k\not \in Q^{\prime}$, a contradiction to Observation \ref{lem2.4} $(1)$. Hence, $Q[W\cup V(Q^{\prime}_i)]$ is a complete graph for $1\leq i\leq p$.

If $|W|\geq 1$, then $o(Q-W)\geq |W|+1$.
Let $x_1\in W$ and $x_2\in V_0(G)$.
Then $\alpha^{\prime}(Q)=\alpha^{\prime}(Q+x_1x_2)$ and $P_k\not\subset Q+x_1x_2$, a contradiction to Observation \ref{lem2.4} $(2)$.
It implies that $|W|=0$. Then $\alpha^{\prime}(Q)=\lfloor\frac{|Q|}{2}\rfloor$ and $Q$ is a clique by Observation \ref{lem2.4} $(1)$.

(2) To the contrary,
suppose that $|Q|$ is even and $|Q|\neq k-1$, which implies $|Q|\leq k-2$. Then $\alpha^{\prime}(Q+xy)=\alpha^{\prime}(Q)$ (as $|Q|$ is even) and $P_k\not \subset Q+xy$ (as $|Q|\leq k-2$), where $x\in Q$ and $y\in V_0(G)$, a contradiction to Observation \ref{lem2.4} $(2)$. Hence, $|Q|$ is odd or $|Q|=k-1$ with odd $k$.
\end{proof}

\begin{lem}\label{lem5}
Let $G$ be a connected graph of order $n\geq 7$ and $\delta(G)\geq 2$. If $G$ is $P_7$-free and $G$ contains $P_5$ as a subgraph, then $G$ is book-structural, or $G$ is isomorphic to each of $F_i,\ \Delta F_i$ or $\Delta_+ F_i$ for some $i\geq 3$, or $G$ is isomorphic to $F^2_{i,j}$ for some $i+j\geq 3$ and $i,j\geq 1$.
\end{lem}

\begin{proof}

Select the longest path $P$ in $G$, say $P=x_1x_2\dots x_k$, then $N(x_1)\setminus V(P)=N(x_k)\setminus V(P)=\emptyset$. Since $G$ is $P_7$-free and $G$ contains $P_5$ as a subgraph, we have $k=5$ or $6$. Let $M$ be the set of vertices in $V(G)\setminus V(P)$ with at least one neighbor in $V(P)$.

$\mathbf {Case~1}$. $k=5$.

If there exists $x\in M$ such that $x_2\in N(x)$ or $x_4\in N(x)$, without loss of generality, assume that $x_2\in N(x)$, then $x_3\notin N(x)$ and $N(x)\subset V(P)$ as $P$ is the longest path in $G$. Hence $N(x)=\{x_2,x_4\}$ by $\delta (G)\geq 2$. Recall that $d(x_3)\geq 2$. If $d(x_3)=2$, then $G$ is book-structural. If $d(x_3)\geq 3$, let $v\in N(x_3)\setminus \{x_2,x_4\}$, then $vx_3x_2xx_4x_5$ or $vx_3x_4xx_2x_1$ is a copy of $P_6$ in $G$, a contradiction to $k=5$. Thus, if there exists $x\in M$ such that $x_2\in N(x)$ or $x_4\in N(x)$, then $G$ is book-structural.

Next, we can assume $N(x)\cap V(P)=\{x_3\}$ for each $x\in M$.
By $\delta(G)\geq 2$ and $P$ is the longest path of $G$, one has that there exists another vertex $y\in V(G)\setminus (V(P)\cup \{x\})$ and $N(y)=\{x,x_3\}$. Considering $P_5=yxx_3x_4x_5$, we have $N(x_1)=\{x_2,x_3\}$. Similarly, $N(x_5)=\{x_4,x_3\}$. Hence, $N(x)=\{x_3,y\}$, $N(x_2)=\{x_1,x_3\}$ and $N(x_4)=\{x_3,x_5\}$. Let $v\in N(x_3)\setminus V(P)$ and $u\in N(v)\setminus \{x_3\}$. Considering $P_5=uvx_3x_4x_5$, we have $N(u)=\{v,x_3\}$ and $N(v)=\{u,x_3\}$. Hence, $G$ is isomorphic to $F_i$, for some $i\geq 3$.

$\mathbf {Case~2}$. $k=6$.

Suppose that there exists $x\in M$, such that $|N(x)\cap V(P)|\geq 3$. Then there exists a copy of $P_7$ in $G$, a contradiction to $k=6$. Thus, we only consider the case of $|N(x)\cap V(P)|\leq 2$ for any $x\in M$.

Suppose that there exists $x\in M$ such that $|N(x)\cap V(P)|=2$. Then $\{x_2,x_4\}\subset N(x)$, or $\{x_3,x_5\}\subset N(x)$, or $\{x_2,x_5\}\subset N(x)$. As the discussions of $\{x_2,x_4\}\subset N(x)$ and $\{x_3,x_5\}\subset N(x)$ are similar, we only consider the former. Then $\{x_2,x_4\}=N(x)$. If $d(x_3)=2$, then $G$ is book-structural. If $d(x_3)\geq 3$ and there exists $v\in N(x_3)\setminus \{x_2,x_4,x_5,x_6\}$, then $vx_3x_2xx_4x_5x_6$ is a copy of $P_7$ in $G$, a contradiction. If $x_3x_5\in E(G)$, then $x_6x_5x_3x_4xx_2x_1$ is a copy of $P_7$ in $G$, a contradiction. If $x_3x_6\in E(G)$, then $x_5x_6x_3x_4xx_2x_1$ is a copy of $P_7$ in $G$, a contradiction. So if $\{x_2,x_4\}\subset N(x)$ or $\{x_3,x_5\}\subset N(x)$, then $G$ is book-structural. Assume that $\{x_2,x_5\}\subset N(x)$. Then $\{x_2,x_5\}=N(x)$. There exists a path $P_7$ in $G$ if $N(x_1)\cap\{x_3,x_4,x_6\}\neq \emptyset$. Hence, $N(x_1)=\{x_2,x_5\}$. Then $G$ is book-structural.

By the discussion above, we can assume that $|N(v)\cap V(P)|=1$ for any $v\in M$. Then $N(v)\cap V(P)=\{x_3\}$ or $N(v)\cap V(P)=\{x_4\}$. According to symmetry, we only consider that there exists $x\in M$ such that $N(x)\cap V(P)=\{x_3\}$, then there exists $y\in N(x)\setminus V(P)$ and $N(y)=\{x,x_3\}$. Considering $P_6=yxx_3x_4x_5x_6$, we have $N(x_1)=\{x_2,x_3\}$. Then $N(x_2)=\{x_1,x_3\}$ and $N(x)=\{x_3,y\}$. If there exists $w\in N(x_3)\setminus V(P)$ and $u\in N(w)\setminus \{x_3\}$. Considering $P_6=uwx_3x_4x_5x_6$, we have $N(u)=\{w,x_3\}$ and $N(w)=\{u,x_3\}$. If there exists $v_1\in N(x_4)\setminus V(P)$ and $u_1\in N(v_1)\setminus \{x_4\}$. Considering $P_6=x_1x_2x_3x_4v_1u_1$, we have $N(u_1)=\{v_1,x_4\}$ and $N(v_1)=\{u_1,x_4\}$. Similarly, $N(x_6)=\{x_4,x_5\}$ and $N(x_5)=\{x_4,x_6\}$. Hence, $G\cong F^2_{i,j}$ for some $i+j\geq 3$ . If $N(x_4)\setminus V(P)=\emptyset$, then $G[\{x_3,x_4,x_5,x_6\}]$ is a subgraph of $K_4$. Hence, $G$ is book-structural, or $G\cong \Delta F_i$ or $G\cong \Delta_+ F_i$ for some $i\geq 3$.
\end{proof}

\begin{lem}\label{lem10}

Let $G$ be a $(P_k+tP_2)$-saturated graph with $k\geq 6$ and $t\geq 1$. If $V_0(G)\neq \emptyset$, then $G$ is not book-structural.

\end{lem}

\begin{proof}
To the contrary, suppose that $G$ is book-structural. Let $u,v\in V(G)$ such that $N(u)=N(v)=\{x,y\}$. Let $w\in V_0(G)$ and $H$ be a copy of $P_k+tP_2$ in $G+wx$. By Lemma \ref{lem4}, $wx\in E(H)$. By $k\geq 6$, $\{u,v\}\not\subset V(H)$. We may therefore assume that $u\notin V(H)$. Thus, $H-wx+ux$ is a copy of $P_k+tP_2$ in $G$, a contradiction.
\end{proof}

Let $G$ be a graph, for any vertices $u,v\in G$, denoted by $S_{uv}$ the set of the paths with endpoints $u$ and $v$. The distance between $u$ and $v$ is denoted as $dist(u,v)$, where $dist(u,v)=$ min$\{|P|:P\in S_{uv}\}$. Define $dist(u,v)=\infty$, if $S_{uv}=\emptyset$. Let $G_1,G_2\subset G$. The distance between $G_1$ and $G_2$ is denoted as $dist(G_1,G_2)$, where $dist(G_1,G_2)=$ min$\{dist(u,v):u\in V(G_1),v\in V(G_2)\}$.

\section{The main results}

\begin{lem}\label{lem6}

Let $G$ be a $(P_7+mP_2)$-saturated graph with $|V_0(G)|\geq 2$. Let $Q=Q_1+Q_2+\dots +Q_l$, where $Q_1,Q_2,\dots ,Q_l$ are all the nontrivial components of $G$. If $|Q|\geq 2m+7$ and $P_7\subset Q_i$, for all $1\leq i\leq l$, then

  \item (1) $G$ is a $(m+3)P_2$-saturated graph,
  \item (2) $|E(G)|>3m+25$.

\end{lem}

\begin{proof}

(1) Since $G$ is a $(P_7+mP_2)$-saturated graph, $G+e$ contains $P_7+mP_2$ for any edge $e\in E(\overline G)$. It follows that $G+e$ contains $(m+3)P_2$. If $G$ is not a $(m+3)P_2$-saturated graph, then $G$ contains $(m+3)P_2$. Since $|Q|\geq 2m+7$, there exists a vertex $v$ not in the copy of $(m+3)P_2$, say $v\in Q_1$.

Since $G$ is a $(P_7+mP_2)$-saturated graph with $|V_0(G)|\geq 2$, we have $V_1(G)= \emptyset$ by Lemma \ref{lem4}. Hence, $\delta (Q)\geq 2$.

 Let $M$ be a copy of $(s+3)P_2$ in $G$ such that

\begin{itemize}
  \item $(i)$ $V(Q_1)-V(M)\neq \emptyset$, and
  \item $(ii)$ subject to $(i)$, $s$ achieves the maximum.
\end{itemize}
Then, $s\geq m$. Since $P_7\subset Q_1$, $Q_1$ contains a copy of $3P_2$. This together with the choice of $M$ implies $|E(M)\cap E(Q_1)|\geq 3 $. From $(i)$, we know that $|V(Q_1)-V(M)|\geq 1$. Since $Q_1$ is connected, there exists a path $vu_1v_1\subset Q_1$, where $v\in V(Q_1)\setminus V(M)$ and $u_1v_1\in E(M)$. If $v_1$ is incident with an edge $u_2v_2$ in $M$, then we have $Q_1[\{v,u_1,v_1,u_2,v_2\}]$ contains a $P_5$. Otherwise, since $d(v_1)\geq 2$, there exists a vertex in $V(Q_1)\setminus V(M)$, say $z$ (possibly $z=v$), adjacent to $v_1$. Since $Q_1$ is connected, there exists an edge in $M$, say $u_2v_2$ such that $2\leq dist(M\setminus \{u_1v_1\},Q_1[\{v,z,u_1,v_1\}])=dist(u_2v_2,Q_1[\{v,z,u_1,v_1\}])< \infty$. Then there exists a copy of $P_5$, say $L$, in $Q_1[\{v,z,u_1,v_1,u_2,v_2\}]$. Let $L=x_1x_2\dots x_5$. In addition, $L$ has at most two $P_2$s used in $E(M)\cap E(Q_1)$.

Let $H=L\cup H_1$ be a subgraph of $Q$ with $H_1\cong (s+1)P_2$ and $V(L)\cap V(H_1)=\emptyset$. Let $D=\{v\in H_1: dist(v,V(L))=dist(H_1,V(L))\}$. Since $Q_1$ is connected, $D\neq \emptyset$. Considering $x\in D$ and $N_{H_1}(x)=\{y\}$. Denote $F_L(x)=\{v\in L: dist(v,x)=dist(L,x)\}$. If $x_1$ or $x_5\in F_L(x)$, $G$ contains a copy of $P_7+sP_2$. Since $s\geq m$, $G$ contains a copy of $P_7+mP_2$, a contradiction. So $x_1,x_5\notin F_L(x)$.

We only need to consider the following two cases by the choice of $M$.

$\mathbf {Case~1}$. $|(V(Q_1)\setminus V(H))\cap N(x_1)|\geq 1$ or $|(V(Q_1)\setminus V(H))\cap N(x_5)|\geq 1$.

Without loss of generality, we can assume $|(V(Q_1)\setminus V(H))\cap N(x_5)|\geq 1$, so there exists a vertex $v\in (V(Q_1)\setminus V(H))\cap N(x_5)$. Then $N(x_1)\subset \{x_2,x_3,x_4,x_5,v\}$ and $N(v)\subset V(L)$. If $u\in F_L(x)$, then $x\in N(u)$. Otherwise, it is a contradiction to $(ii)$. There exists no vertex $u\in V(Q_1)\setminus (V(H)\cup \{x\})$, such that $uw\in E(Q_1)$, where $w\in V(L)\cup \{v\}$. Otherwise, it is a contradiction to $(ii)$.

Suppose that $x_2\in F_L(x)$. Then $G$ contains a copy of $P_7+sP_2$ with $P_7=yxx_2x_3x_4x_5v$, a contradiction to that $G$ is a $(P_7+mP_2)$-saturated graph. So $x_2\notin F_L(x)$, which implies $N(x_2)\subset V(L)\cup \{v\}$. And $N(x_5)\subset V(L)\cup \{v\}$.
Suppose that $x_3\in F_L(x)$ or $x_4\in F_L(x)$, as the detail discussions are similar, we only consider $x_4\in F_L(x)$.
Assume that $x_3\in N(y)$. Then $G$ contains a copy of $P_7+sP_2$, a contradiction. Then $x_3\notin N(y)$. Assume that $u\in N(y)\setminus \{x,x_4\}$. Then $P_7=x_1x_2x_3x_4xyu$ and $P_2=x_5v$. Hence, $G$ contains a copy of $P_7+sP_2$, a contradiction. Then $N(y)=\{x,x_4\}$. Moreover, $N(x)=\{y,x_4\}$. Replacing $x_1x_2x_3x_4x_5v$ with $x_1x_2x_3x_4xy$ and $xy$ with $x_5v$, then $N(v)=\{x_5,x_4\}$. Similarly, $N(x_5)=\{x_4,v\}$.
Since $N(x_1)\subset \{x_2,x_3,x_4\}$, one has that $N(x_1)=\{x_2,x_4\}$ or $N(x_1)=\{x_2,x_3\}$ or $N(x_1)=\{x_2,x_3,x_4\}$.
If there exists an edge $u_2v_2\in E(H_1)\setminus \{xy\}$ that adjacent to $x_4$, then replacing $x_1x_2x_3x_4x_5v$ with $x_1x_2x_3x_4u_2v_2$. Therefore, $N(v_2)=\{u_2,x_4\}$. Similarly, $N(u_2)=\{v_2,x_4\}$.
As for $x_3$, if there exists an edge $u_3v_3\in E(H_1)\setminus \{xy\}$ that adjacent to $x_3$, then replacing $x_1x_2x_3x_4x_5v$ with $vx_5x_4x_3u_3v_3$. Therefore, $N(v_3)=\{u_3,x_3\}$ and $N(u_3)=\{v_3,x_3\}$. Hence, $Q_1$ is isomorphic to $F^2_{j,l}$, for some $j+l\geq 3$ and $j,l\geq 1$, a contradiction to $P_7\subset Q_1$. So $N(x_3)\subset \{x_1,x_2,x_4\}$, then $Q_1[\{x_1,x_2,x_3,x_4\}]$ is a subgraph of $K_4$. Hence, $Q_1$ is book-structural, or $Q_1$ is isomorphic to $\Delta F_i$ or $\Delta_+ F_i$ for some $i\geq 3$, a contradiction to $P_7\subset Q_1$.

$\mathbf {Case~2}$. $|(V(Q_1)\setminus V(H))\cap N(x_1)|=0$ and $|(V(Q_1)\setminus V(H))\cap N(x_5)|=0$.

We have $N(x_1)\subset \{x_2,x_3,x_4,x_5\}$ and $N(x_5)\subset \{x_1,x_2,x_3,x_4\}$.
Suppose that $x_2\in F_L(x)$ or $x_4\in F_L(x)$, as the discussions are similar, we only consider $x_4\in F_L(x)$.
Then $N(x_5)=\{x_2,x_4\}$ and $N(x_1)=\{x_2,x_4\}$. Otherwise, $G$ contains a copy of $P_7+sP_2$, a contradiction. But in this situation, $Q_1$ is book-structural, a contradiction. Hence, $x_2,x_4\notin F_L(x)$.
If $x_3\in F_L(x)$, then let $dist(x_3,x)=t$. By $(ii)$, we have $t\leq 3$. If $t=3$, then there exists a path $x_3ux\in S_{x_3x}$ with $u\in V(Q_1)\setminus V(H)$, the discussion is similar to Case $1$ by replacing $x_1x_2x_3x_4x_5v$ with $x_1x_2x_3uxy$. So we only discuss $t=2$, which implies $x_3\in N(x)$. We have $N(x_1)=\{x_2,x_3\}$, otherwise, $G$ contains a copy of $P_7+sP_2$, a contradiction. Replacing $x_1x_2x_3x_4x_5$ with $x_2x_1x_3x_4x_5$, then $N(x_2)=\{x_1,x_3\}$. Similarly, $N(x_5)=\{x_3,x_4\}$, $N(y)=\{x_3,x\}$, $N(x_4)=\{x_3,x_5\}$ and $N(x)=\{x_3,y\}$. As for $x_3$, if there exists a vertex $w\in V(Q_1)\setminus V(H)$ that adjacent to $x_3$, then by $(ii)$ and $\delta (Q)\geq 2$ there exists an edge $u_3v_3\in E(H_1)\setminus \{xy\}$ that adjacent to $w$. It can be transformed into Case $1$ by replacing $x_1x_2x_3x_4x_5v$ with $x_1x_2x_3wu_3v_3$. If there exists an edge $u_2v_2\in E(H_1)\setminus \{xy\}$ that adjacent to $x_3$, then replacing $x_1x_2x_3x_4x_5$ with $x_1x_2x_3u_2v_2$. Therefore, $N(v_2)=\{u_2,x_3\}$. Similarly, $N(u_2)=\{v_2,x_3\}$. Hence, $G$ is isomorphic to $F_i$, for some $i\geq 3$, a contradiction to $P_7\subset Q_1$. Hence, $G$ is a $(m+3)P_2$-saturated graph.

(2) Suppose that $|E(G)|\leq 3m+25$. By $(1)$, $G$ is a $(m+3)P_2$-saturated graph, then $\alpha^{\prime}(Q)=m+2$. By Lemma \ref {lem1}, there exists $Y\subset V(Q)$ such that

\begin{center}
$m+2=\frac{1}{2}(|Q|+|Y|-o(Q-Y))$.
\end{center}

Let $Q^{\prime}_1, Q^{\prime}_2, \dots , Q^{\prime}_p$ be the components of $Q-Y$.

Suppose that there exist two vertices $u,v\in Y\cup V(Q^{\prime}_i)$ such that $uv\notin E(Q)$. Let $Q^{\prime}=Q+uv$. Since $Q$ is a $(m+3)P_2$-saturated graph, $\alpha^{\prime}(Q^{\prime})\geq m+3$. On the other hand,

\begin{center}
$m+2=\frac{1}{2}(|Q|+|Y|-o(Q-Y))=\frac{1}{2}(|Q^{\prime}|+|Y|-o(Q^{\prime}-Y))\geq \alpha^{\prime}(Q^{\prime})\geq m+3$,
\end{center}
a contradiction. Hence, $Q[Y\cup V(Q^{\prime}_i)]$ is a complete graph for $1\leq i\leq p$.

If $Y=\emptyset$, then $Q^{\prime}_i=Q_i$ is a complete graph, $p=l$ and $\delta(Q_i)\geq 6$, for any $1\leq i\leq p$. $2(3m+25)\geq 2|E(Q)|\geq 6(2m+7)=12m+42$. Hence, $m=1$. Then $|E(Q)|\geq |E(K_9)|=36>28$, a contradiction to $|E(G)|\leq 3m+25$. Hence, $Y\neq \emptyset$. So we have $l=1$, then $Q=Q_1$. Let $x\in Y$ and $v\in V_0(G)$. By Lemma \ref{lem4}, we have $N[x]\cup \{v\}\subset V(H)$, where $H$ is a copy of $P_7+mP_2$ in $G+xv$. Then $2m+7+1\leq |Q|+1=|N[x]\cup \{v\}|\leq |H|=2m+7$, a contradiction. Hence, $|E(G)|>3m+25$.
\end{proof}

\begin{thm}\label{thm1}
Suppose that $n$ and $t$ are two positive integers with $n$ sufficiently large. Let $G$ be a $(P_7+tP_2)$-saturated graph with $|V_0(G)|\geq 2$ and $|E(G)|\leq 3t+25$. Then $|E(G)|=3t+25$ and $G=K_8+cK_3+F+\overline{K}_{n-3t_3-|F|-8}$, where $c$ is a non-negative integer, $F$ is the disjoint union of fans with order at least $7$ and $\alpha^{\prime}(cK_3+F)=t-1$.
\end{thm}

\begin{proof}

Since $G$ is a $(P_7+tP_2)$-saturated graph with $|V_0(G)|\geq 2$, one has that $V_1(G)=\emptyset$ and $P_7\subset G$. By Lemma \ref{lem15}, any component with order $i$ in $G$ is a complete graph and there exists no $K_2$ or $K_4$ in $G$, where $i\in\{3,5,6\}$.
 Let $W$ be a component of $G$ with $|W|\geq 7$.
\begin {claim}\label{cl} If $W$ is not isomorphic to $F_i$ for some $i\geq 3$, then $P_7\subset W$.
\end {claim}

\noindent $\mathbf {Proof\ of\ Claim\ \ref{cl}.}$ Assuming $P_k=x_1x_2\dots x_k$ is the longest path of $W$. By Lemma \ref{lem4}, $\delta (W)\geq 2$. To the contrary, assume $k\leq 4$. If $k=2$ or $3$, then $Q$ is an edge or a star, a contradiction to $\delta (W)\geq 2$. If $k=4$, then $C_4\subset W[\{x_1,x_2,x_3,x_4\}]$. Since $|Q|\geq 7$, then there exists a vertex $v$ in $W\setminus \{x_1,x_2,x_3,x_4\}$ such that $dist(v, Q[\{x_1,x_2,x_3,x_4\}])=2$. Then $P_5\subset W[\{v,x_1,x_2,x_3,x_4\}]$, a contradiction. Hence, $k\geq 5$, which implies $P_5\subset W$.

If $P_7\not \subset W$, by Lemma \ref{lem5}, $W$ is book-structural, or $W$ is isomorphic to each of $F_i,\ \Delta F_i$ or $\Delta_+ F_i$ for some $i\geq 3$, or $W$ is isomorphic to $F^2_{i,j}$ for some $i+j\geq 3$ and $i,j\geq 1$.
By Lemma \ref{lem10}, $W$ is not book-structural. Notice that joining the center vertex in $\Delta fan$ to a vertex in $V_0(G)$ does not increase the matching number or get a new $P_7$ in $G$, then there exists no $\Delta fan$ in $G$. Similarly, there exists no $ffan$ or $\Delta_+ fan$ in $G$. Hence, if $W$ is not isomorphic to $F_i$ for some $i\geq 3$, then $P_7\subset W$.
$\hfill\blacksquare$

Let $G=G^{\prime}+t_3K_3+t_5K_5+t_6K_6+F$, where $t_i$ is the number of $K_i$ of $G$, $i=3,5,6$ and $F$ is the disjoint union of fans with order at least $7$.
By Claim \ref{cl}, there exists $P_7$ in all nontrivial components of $G^{\prime}$ and $G^{\prime}$ is not book-structural.
Denoted by $F_c$ the number of $fan$ in $F$. We have
\begin {center}
$t_3+2t_5+3t_6+\frac{1}{2}(|F|-F_c)\leq t-1$.
\end {center}
 Let
\begin {center}
$t^{\prime}=t-(t_3+2t_5+3t_6+\frac{1}{2}(|F|-F_c))$,
\end {center}
then $t^{\prime}\geq 1$ and $|G^{\prime}|=n^{\prime}=n-(3t_3+5t_5+6t_6+|F|)$.

Note that $G^{\prime}$ is a $(P_7+t^{\prime}P_2)$-saturated graph, $|V_0(G)|\geq 2$ and $\delta (Q^{\prime})\geq 2$, where $Q^{\prime}=G^{\prime}-V_0(G)$. Since
\begin {center}
 $|E(G^{\prime})|=|E(G)|-(3t_3+10t_5+15t_6+\frac{3}{2}(|F|-F_c))\leq 3t^{\prime}+25-(4t_5+6t_6)\leq 3t^{\prime}+25$,
\end {center}
by Lemma \ref{lem6}, we have $|Q^{\prime}|\leq 2t^{\prime}+6$. Thus, there is no copy of $P_7+t^{\prime}P_2$ in $G^{\prime}+e$, for any $e\in \overline {Q^{\prime}}$ (if any), a contradiction to that $G^{\prime}$ is a $(P_7+t^{\prime}P_2)$-saturated graph. Then $|E(\overline {Q^{\prime}})|=0$. So $Q^{\prime}$ is a complete graph. Since there is a copy of $P_7+t^{\prime}P_2$ in $G+\{uv\}$ with $u\in Q^{\prime}$ and $v\in V_0(G^{\prime})$, one has that $|Q^{\prime}|\geq 2t^{\prime}+7-1=2t^{\prime}+6$. Thus, $Q^{\prime}=K_{2t^{\prime}+6}$. Notice that
\begin {center}
$|E(Q^{\prime})|=|E(G^{\prime})|=\frac{1}{2}(2t^{\prime}+6)(2t^{\prime}+5)\leq 3t^{\prime}+25$.
\end {center}
 It follows that $t^{\prime}=1$ and $Q^{\prime}=K_8$. Hence, $t_5=0$ and $t_6=0$. Consequently,
\begin {center}
$G=K_8+cK_3+F+\overline{K}_{n-3t_3-|F|-8}$,
\end {center}
where $\alpha^{\prime}(cK_3+F)=t-1$, $c$ is a non negative integer and $F$ is the set of fan with order at least $7$.
\end{proof}

\begin{lem}\label{lem8}
Let $G$ be a $(P_7+tP_2)$-saturated graph with $|V_0(G)|=0$. If the number of tree components in $G$ is at least two, then the order of each tree component in $G$ is at least $14$.
\end{lem}

\begin{proof}

Let $T$ be a tree component of $G$ and $S$ be the set of vertices of $T$ with degree at least $3$. By Lemma \ref{lem2}, there exists no vertex with degree $2$ in $T$. If $|S|\geq 6$, then $|T|\geq 14$.

Assuming $P_k=x_1x_2\dots x_k$ is the longest path in $T$. Suppose that $k=5$ or $4$, then $P_7\not\subset T+x_2x_4$ and $\alpha^{\prime}(T+x_2x_4)\leq \alpha^{\prime}(T)$, a contradiction to Observation \ref{lem2.4} $(1)$. So $k\neq 5$ or $4$. Suppose that $k=3$ or $2$, then $T$ is a star or an isolated edge. Let $T_1$ be another tree component of $G$ with the longest path $y_1y_2\dots y_l$. If $l\leq 3$, then $P_7\not\subset T+T_1+x_2y_2$ and $\alpha^{\prime}(T+T_1+x_2y_2)\leq \alpha^{\prime}(T+T_1)$, a contradiction to Observation \ref{lem2.4} $(2)$. So we can assume $l\geq 6$. Let $H$ be the $P_7+tP_2$ in $G+x_2y_3$. Then $x_2y_3\in E(H)$ and $x_1\notin V(H)$. Hence, $x_2y_3$ is contained in the $P_7$ of $H$. Replace $y_3x_2x_1$ with $y_3y_2y_1$, there exists $P_7+tP_2$ in $G$, a contradiction. So $k\neq 3$ or $2$. Hence, $k\geq 6$.

If $k\geq 8$, then we have $|S|\geq 6$. If $k=7$ and $|S|=5$, then $\alpha^{\prime}(T)\leq \alpha^{\prime}(T+x_3v)$ and $\alpha_7(T)\leq \alpha_7(T+x_3v)$, a contradiction to Observation \ref{lem2.4} $(1)$, where $v$ is a leaf of $T$ with $N_{T}(v)=\{x_4\}$. Hence, $|S|\geq 6$. Suppose that $k=6$. If $|S|=4$, then $P_7\not\subset T+x_3x_5$ and $\alpha^{\prime}(T+x_3x_5)\leq \alpha^{\prime}(T)$, a contradiction to Observation \ref{lem2.4} $(1)$. If $|S|=5$, then assuming $S=\{x_2,x_3,x_4,x_5,v\}$. If $v\in N(x_3)$, then $P_7\not\subset T+x_3x_5$ and $\alpha^{\prime}(T+x_3x_5)\leq \alpha^{\prime}(T)$, a contradiction to Observation \ref{lem2.4} $(1)$. If $v\in N(x_4)$, then $P_7\not\subset T+x_2x_4$ and $\alpha^{\prime}(T+x_2x_4)\leq \alpha^{\prime}(T)$, a contradiction to Observation \ref{lem2.4} $(1)$. Hence, $|S|\geq 6$. Then $|T|\geq 14$.
\end{proof}

\begin{proof}[$\mathbf {Proof\ of\ Theorem\ \ref{thm9}}$]

Suppose $G$ is $(P_7+tP_2)$-saturated.

If $|V_0(G)|\geq 2$, by Theorem \ref{thm1}, then sat$(n,P_7+tP_2)=3t+25$ and $G=K_8+cK_3+F+\overline{K}_{n-3t_3-|F|-8} \in Sat(n,P_7+tP_2)$ for $n\geq \frac{14}{13}(3t+25)$, where $c$ is a non-negative integer, $F$ is the disjoint union of fans with order at least $7$ and $\alpha^{\prime}(cK_3+F)=t-1$.

If $|V_0(G)|=1$, $|E(G)|=\frac{1}{2}\sum_{v\in G}d(v)> n-1\geq $ min$\{3t+25,n-\lfloor \frac{n}{14} \rfloor\}$.

If $|V_0(G)|=0$, suppose $|E(G)|<n-\lfloor \frac{n}{14} \rfloor$. Let $G=G_0+(T_1+T_2+\dots +T_k)$ be a $(P_7+tP_2)$-saturated graph, where $T_1,T_2,\dots ,T_k$ are all the tree components of $G$. Hence,

\begin{center}
$|E(G)|=|E(G_0)|+\sum ^k_{i=1}|E(T_i)|\geq |G_0|+\sum ^k_{i=1}(|T_i|-1)=|G|-k=n-k$.
\end{center}

We have $k>\lfloor \frac{n}{14} \rfloor$ as $|E(G)|<n-\lfloor \frac{n}{14} \rfloor$. If $k\geq 2$, $|T_i|\geq 14$ by Lemma \ref{lem8}. Thus, $n\geq 14k$, a contradiction. If $k=1$, $|E(G)|\geq n-1> n-\lfloor \frac{n}{14} \rfloor$ as $n\geq \frac {14}{5}t+27$. Hence, $|E(G)|\geq n-\lfloor \frac{n}{14} \rfloor$.

On the other hand, set $n=14q+r$, where $q=\lfloor \frac{n}{14} \rfloor$ and $r=0,1,\dots ,13$. Since $n\geq \frac {14}{5}t+27$, we have $t\leq 5q+\lceil \frac {5}{14}(r-27)\rceil \leq 5q+ \lceil \frac {5}{14}(13-27)\rceil = 5q-5$. Consider the graph $H=(q-1)T+T^*$, where $T$ and $T^*$ is shown in Figure \ref{Fig.5} and Figure \ref{Fig.6}, respectively. Forming a $P_7$ requires at most $5$ disjoint $P_2$. And $\alpha _7(H+xy)=5q-5$, where $x,y$ are shown in Figure \ref{Fig.5} and Figure \ref{Fig.6}, respectively. It is obvious that $H$ contains no copy of $P_7$ and $H+e$ contains a copy of $P_7+(5q-5)P_2$ for any $e\in E(\overline H)$. Thus, $H\in Sat(n,P_7+tP_2)$. And $|H|=n-\lfloor \frac{n}{14} \rfloor$. Hence, $sat(n,P_7+tP_2)=n-\lfloor \frac{n}{14} \rfloor$.
\begin{figure}[htbp]
  \centering
  \begin{minipage}[t]{0.45\textwidth}
    \centering
    \includegraphics[width=\textwidth]{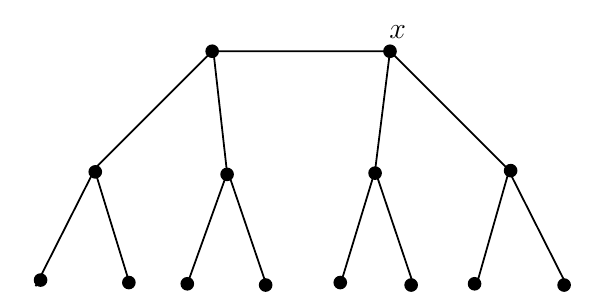} 
     \caption{$T$}\label{Fig.5}
  \end{minipage}
  \hfill
  \begin{minipage}[t]{0.45\textwidth}
    \centering
    \includegraphics[width=\textwidth]{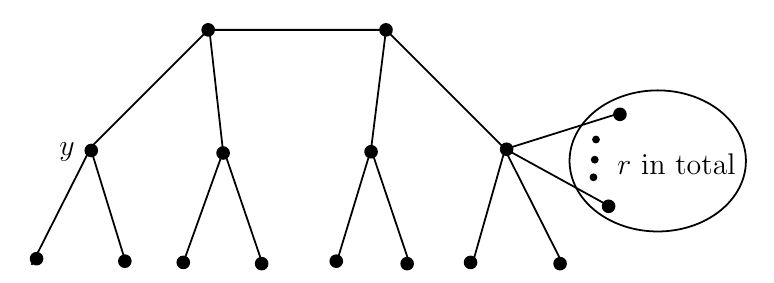} 
    \caption{$T^*$}\label{Fig.6}
  \end{minipage}
\end{figure}

This completes the proof of Theorem \ref{thm9}.
\end{proof}

\section*{Declaration of competing interest}

This paper does not have any conflicts to disclose.

\section*{Acknowledgements}

This paper is supported by the National Natural Science Foundation of China (No.~12331013 and No.~12161141005); and by Beijing Natural Science Foundation (No.~1244047), China Postdoctoral Science Foundation (No.~2023M740207).

\def\polhk#1{\setbox0=\hbox{#1}{\ooalign{\hidewidth
\lower1.5ex\hbox{`}\hidewidth\crcr\unhbox0}}}

\end{document}